\documentclass[10pt]{amsart}
\usepackage{txfonts}
\usepackage[T1]{fontenc}
\usepackage{fancyhdr}
\usepackage{setspace}

\usepackage{amsmath,amsthm,amssymb,hyperref, mdframed}
\usepackage{mathrsfs} 
\usepackage{tikz}
\usepackage[all]{xy}
\usetikzlibrary{decorations.markings}
\usetikzlibrary{backgrounds,shapes}
\usepackage{subfig}
\usepackage{tikz}
\usetikzlibrary{shapes,arrows}
\usetikzlibrary{shapes.geometric, arrows}
\usetikzlibrary{decorations.markings}
\usepackage{subfig}
\usepackage{tabularx}
\usepackage{enumitem}
\usepackage{amsmath,amsthm,amscd}
\usepackage{amssymb,amsfonts}
\usepackage{fancyhdr} 
\usepackage{pgf,tikz}
\usepackage{caption}
\usepackage{tikz-cd}
\usepackage{tikzscale}
\usetikzlibrary{patterns}
\usepackage{rotating}
\usepackage{xcolor}
\usepackage{lscape}
\usepackage{array} 
\usepackage{tabularx}
\usepackage{multirow}
\usepackage{booktabs} 
\usepackage{caption}
\usepackage{xspace}
\usepackage{xfrac}
\usepackage{blindtext}
\usepackage[ruled,vlined]{algorithm2e}

\usetikzlibrary{decorations.markings}
\usetikzlibrary{patterns}
\usetikzlibrary{calc}

\usepackage{color}
\usepackage{hyperref}
\hypersetup{
    colorlinks=true, 
    linktoc=all,     
    linkcolor=blue,  
    citecolor=blue,
    filecolor=blue,
    urlcolor=blue
}

\definecolor{aliceblue}{rgb}{0.9, 0.95, 1.0}

\usepackage{geometry} 

\numberwithin{equation}{section}

\newcommand\Z{{\mathbb Z}}

\DeclareMathOperator\aut{Aut}

\newcommand{\C}{{\mathbb C}}

\newcommand{\pslc}{{\mathrm{PSL}(2,\, \mathbb{C})}}
\newcommand{\pslr}{{\mathrm{PSL}(2,\mathbb{R})}}

\newcommand{\cp}{\mathbb{C}\mathrm{\mathbf{P}}^1}

\newcommand{\shomolzn}{\textnormal{H}_1(S_{g,n},\mathbb{Z})}

\newcommand{\shomolz}{\textnormal{H}_1(S_{g},\mathbb{Z})}

\theoremstyle{plain}                    
\newtheorem{thm}{Theorem}[section]
\newtheorem*{hthm}{Haupt's Theorem}
\newtheorem{thma}{Theorem}
\newtheorem{propa}[thma]{Proposition}
\newtheorem{cora}[thma]{Corollary}

\newtheorem{lem}[thm]{Lemma}
\newtheorem{prop}[thm]{Proposition}
\newtheorem{cor}[thm]{Corollary}
\newtheorem*{thmnn}{Theorem}

\theoremstyle{definition}
\newtheorem{defn}[thm]{Definition}

\newtheorem{ex}[thm]{Example}

\newtheorem{rmk}[thm]{Remark}

\newcommand*{\defeq}{\mathrel{\vcenter{\baselineskip0.5ex \lineskiplimit0pt
                     \hbox{\scriptsize.}\hbox{\scriptsize.}}}%
                     =}

\tikzstyle{rb} = [rectangle, rounded corners, minimum width=3cm, minimum height=1cm, text width=3cm, text centered, draw=black, fill=blue!30]
\tikzstyle{sb} = [rectangle, minimum width=2cm, minimum height=1cm, text width=3cm, text centered, draw=black, fill=violet!30]

\title[Automorphism groups of certain branched structures on surfaces]{On the automorphism groups of certain branched structures on surfaces}

\author{Gianluca Faraco}
\address[]{Mathematisches Institut Rheinische Friedrich-Wilhelms-Universit\"at, Bonn, Germany}
\email{gfaraco@uni-bonn.de}
\email{gianluca.faraco.math@gmail.com}

\begin{document}

\keywords{}
\subjclass[]{}%
\date{\today}
\dedicatory{}

\begin{abstract}
We consider translation surfaces with poles on surfaces. We shall prove that any finite group appears as the automorphism group of some translation surface with poles. As a direct consequence we obtain the existence of structures achieving the maximal possible number of automorphisms allowed by their genus and we finally extend the same results to branched projective structures.
\end{abstract}

\maketitle
\tableofcontents

\section{Introduction} 

\noindent Let $S_g$ be a closed surface of genus $g$ and let $\mathcal{M}_g$ be the moduli space of compact Riemann surfaces homeomorphic to $S_g$. For a Riemann surface $X$, let $\aut(X)$ denote its group of holomorphic automorphisms. If $X$ has genus at least $2$, a classical result by Hurwitz, see \cite{Hur}, states that its group of conformal automorphisms is a finite group with the cardinality being bounded only in terms of the genus; \textit{i.e.} $|\aut(X)\,|\le 84(g-1)$. Riemann surfaces which achieve this bound has been named \textit{Hurwitz surfaces} and the finite groups arising as their groups of automorphisms have been named Hurwitz groups. Hurwitz surfaces are pretty rare as they do not appear in every genus. For instance, there is none in genus two and the so-called \textit{Bolza surface} is the only Riemann surface with the highest possible order of the conformal automorphism group in this genus, see \cite{Bol}. The Hurwitz surface of lowest possible genus appears in genus three and it is known as the \textit{Klein quartic}, see \cite{Kle}. Along the years, a lot of research has been developed to study geometric and algebraic features of these objects, see \cite{CM} and references therein for a survey. For instance, in \cite{Mac}, Macbeath proved that the Hurwitz bound is attained for infinitely many $g\ge2$ and, around the same time, Greenberg proved that \textit{every} finite group can be represented as the automorphism group of some, possibly non-compact, Riemann surface of finite type, see in \cite{GL, GL2}.

\medskip

\noindent A \textit{translation surface} is a Riemann surface $X$ equipped with an additional structure determined by an abelian differential $\omega\in\Omega(X)$.  For a given pair $(X,\omega)$, let $\aut(X,\omega)$ denote the group of holomorphic automorphisms $f$ such that $f^*\omega=\omega$. For a translation surface $(X,\omega)$, we shall define $\aut(X,\omega)$ as the \textit{group of translations}. It naturally follows from the definition that $\aut(X,\omega)\le \aut(X)$ for any translation surface $(X,\omega)$ and it sounds natural to determine under which conditions $\aut(X,\omega)$ is as large as possible in the overall group $\aut(X)$. In their recent work \cite{SPWS}, Schlage-Puchta and Weitze-Schmith\"{u}sen showed that a translation surface $(X,\omega)$ determined by a \textit{holomorphic} differential has at most $4(g-1)$ translations. They call Hurwitz translation surfaces those structures that achieve this bound and characterise them as normal origamis, \textit{i.e.} square-tiled surfaces which arise as regular coverings of the standard torus $\C\,/\,\Z[\,i\,]$. In the present paper we are mainly interested in \textit{translation surfaces with poles}, that is translation surfaces determined by a \textit{meromorphic} abelian differential with poles of positive finite order, see Section \S\ref{sec:tswp} for more details. Our main result is the following

\smallskip

\begin{thma}\label{thm:main}
Every finite group appears as the group of translations of some translation surface with poles. More precisely, let $X$ be a compact Riemann surface of genus $g\ge2$ and let $G$ its group of conformal automorphisms. Then there exists a meromorphic differential $\omega\in\Omega(X)$ with poles of finite order such that $G=\aut(X,\omega)$.
\end{thma}

\noindent It is immediate to observe that we cannot have a similar statement for translation surfaces without poles, that is determined by a holomorphic differential. In fact, Schlage-Puchta and Weitze-Schmith\"{u}sen's earlier work provides a dramatically sharper upper bound for the cardinality of the group of translation. From their work, it is not clear in principle which finite groups appear as the full group of translations of some translation surface without poles, namely determined by a holomorphic differential. In recent times, however, Hidalgo in \cite{HR} has proved that for a finite group of order $n$ there exists a translation surface $(X,\omega)$, where $\omega\in\Omega(X)$ is a holomorphic differential, of genus at least $\left\lceil\,\frac{n+4}{4}\,\right\rceil\le g$. Moreover, such a translation surface can be assumed to be a normal origami. In the same spirit of Greenberg in \cite{GL}, Hidalgo-Morales recently proved in \cite{HM} that any countable group appears as the group of translation of some origami on the Loch Ness monster, a surface of infinite genus and one end, see \cite{RI} for the classification of infinite type surfaces. Our main Theorem above sits somewhere in between these works as it aims to extend Hidalgo's earlier result to translation surfaces arising from abelian differentials which are meromorphic, but not holomorphic, on a Riemann surface. As a direct consequence we obtain that the Hurwitz bound is sharp for translation surfaces with poles. 

\smallskip


\begin{cora}\label{cor:main}
Let $X$ be a Hurwitz surface. Then there exists a meromorphic differential $\omega\in\Omega(X)$ such that $|\aut(X,\omega)\,|=84(g-1)$.
\end{cora}

\smallskip

\subsection{Groups of translations}\label{ssec:got} The lack of a sharper upper bound for the cardinality of $\aut(X,\omega)$ is related to the existence of translation structures with poles on the Riemann sphere. Since there are no non-trivial holomorphic differentials on $\cp$, for a translation surface $(X,\omega)$, the space $(X,\omega)/\aut(X,\omega)$ is a translation surface of positive genus whenever $\omega$ is a holomorphic differential. On the other hand, since there are meromorphic differentials on $\cp$, which always determine translation structures with poles, the space $(X,\omega)/\aut(X,\omega)$ may very well have genus zero when $\omega$ is meromorphic. We shall say that a translation surface $(X,\omega)$ is \textit{large} if it has a \textit{large group of translations} $\aut(X,\omega)$, that means the space $(X,\omega)/\aut(X,\omega)$ is the Riemann sphere equipped with some abelian differential $\xi$; necessarily with poles. 

\begin{propa}\label{prop:autchar} Let $(X,\omega)$ be a translation surface, possibly with poles, of genus $g\ge2$. If the group of translations $\aut(X,\omega)$ is not large, then \begin{equation}  |\aut(X,\omega)\,|\le 4g-4. \end{equation} \end{propa}

\smallskip



\subsection{Periods and automorphisms} For a translation surface $(X,\omega)$, its \textit{period character} is a representation 
\begin{equation}\label{eq:per}
    \chi:\shomolzn\longrightarrow \C \quad \text{ defined as }\quad \gamma \longmapsto\int_\gamma \omega,
\end{equation} where $n\ge0$ is the number of poles of $\omega$ on $X$. For a peripheral loop $\gamma$, \textit{i.e.} a simple closed loop around a puncture, the period $\chi(\gamma)$ is $2\pi i$ times the residue of $\omega$ at the puncture. The problem of determine which representations appear as the period character of some abelian, possibly meromorphic, differential has been considered by \cite{HO, KM, CFG}.

\smallskip



\noindent It is natural to determine when a representation $\chi$ as in \eqref{eq:per} can be realised as the period character of some translation surface with large group of translations. The following holds.

\begin{propa}\label{prop:charlarge}
Let $(X,\omega)$ be a large translation surface. Then its period character $\chi_\omega$ factors through a representation $\chi_\xi$, where $\xi\in\Omega(\cp)$. Conversely, given $\xi\in\Omega(\cp)$ and its period character $\chi_\xi$, consider $\Gamma=\textnormal{Im}(\,\chi_\xi\,)$. For any $H\lhd \Gamma$ of finite index, there is a large translation surface $(X,\omega)$ with group of translation $\aut(X,\omega)\cong G\cong\Gamma/H$.
\end{propa}

\smallskip 

\noindent Therefore, a representation $\chi:\shomolzn\longrightarrow \C$ can be realised as the period of some large translation surface if it factors through the period character of some meromorphic differential on the Riemann sphere. 

\subsection{Branched projective structures} Translations surfaces, either with or without poles, belong to a much larger family of structures known as \textit{branched projective structures}, that is geometric structures locally modeled on the Riemann sphere with its group of conformal automorphisms $\pslc$, see sub-section \S\ref{ssec:tswpasbps} for a more detailed definition. Any branched projective structure, say $(X,\,\sigma)$, determines a well-defined underlying Riemann surface $X$. In fact, any such a structure can be seen as the choice of some special covering for a Riemann surface in the sense of Gunning \cite{GU}, that is an open cover for which transition functions are not just local biholomorphisms, but they belong to some more restricted group of transformations of the Riemann sphere $\cp$. Generally, a conformal automorphism does not need to preserve this special covering and hence it makes sense to ask how many non-trivial conformal automorphisms do preserve the projective structure. Let us denote with $\aut(X,\,\sigma)$ the group of \textit{projective automorphisms} of $(X,\,\sigma)$. As another immediate consequence of our Theorem \ref{thm:main} we obtain the following

\begin{cora}\label{cor:bps}
Let $X$ be a Hurwitz surface. Then there exists a branched projective structure $\sigma$ on $X$ such that $|\aut(X,\sigma)\,|=84(g-1)$.
\end{cora}

\noindent This corollary extends the earlier result of the author obtained with L. Ruffoni, see \cite{FR}, in the special case of \textit{unbranched} projective structures.

\subsection{Isosymmetric loci of Riemann surfaces and beyond} We would like to conclude this introduction with some additional and motivational remarks. For a finite group $G$, our methods permit to find out a Riemann surface $X$ and a meromorphic differential $\omega$ such that $G\cong\aut(X,\omega)$. A much more subtle question, however, is to determine whether a finite group $G$ appear as the full group of translations of some translation surface with poles and prescribed genus $g$. Of course, as a necessary condition $G$ must appear as the group of conformal automorphisms of some compact Riemann surface of genus $g$. This raises the question of how to single out those groups that appear as the full automorphism group of a genus $g$ Riemann surface. For a finite group $G$ acting on a topological surface $S_g$ of genus $g\ge2$, we may define the $G-$\textit{isosymmetric locus} inside $\mathcal M_g$ as the subspace of those compact Riemann surfaces admitting a $G-$action of given ramification type. These isosymmetric loci have been studied in literature and a complete classification is known for low-genus surfaces, \textit{i.e.} $g=2,3$, see \cite{KuNa, MSSV}. It is natural to introduce similar loci inside the moduli space $\mathcal{H}_g(m_1,\dots,m_k;-p_1,\dots,-p_n)$ of pairs $(X,\omega)$, where $X\in\mathcal M_g$ and $\omega\in\Omega(X)$ has $k$ zeros of orders $m_1,\dots,m_k$ and possibly $n$ poles of orders $p_1,\dots,p_n$. The study of geometry and topology of these loci inside the moduli spaces of translation surfaces turns out to be a challenging question worth of interest. In fact, these moduli spaces are known to be complex orbifolds whose singular points correspond to translation surfaces with non-trivial symmetries. 

\smallskip

\noindent A similar question can be posed for branched projective structures. More precisely: When does a finite group a $G$ appear as the full group of projective automorphisms of a branched projective structure on a Riemann surface of prescribed genus $g$? 
\noindent Even in this case we may introduce $G-$isosymmetric loci inside the space $\mathcal{B\,P}_g(m_1,\dots,m_k)$ of branched projective structures with $k$ singularities of order $m_i\ge0$. In particular, if $m_i=0$ then the moduli space comprises all unbranched projective structures on $S_g$. Again, the geometry and topology of $G-$isosymmetric loci is not known and it is certainly worth of interest. 
For instance, Francaviglia-Ruffoni in \cite{FrRu} studied a certain locus of hyperelliptic branched projective structures in the context of the classical Riemann-Hilbert problem for $\mathfrak{sl}_2$-systems.



\subsection{Organisation of the paper} The present paper is organised as follows. Section \S\ref{sec:tswp} we recall the necessary background about translation surfaces and their periods characters. Section \S\ref{sec:agtswp} we shall prove our Theorem \ref{thm:main} and Propositions \ref{prop:autchar} and \ref{prop:charlarge}. Corollary \ref{cor:main} will follow as a direct consequence. Finally, we shall introduce in Section \S\ref{sec:autbps} branched projective structures on closed surfaces and we provide a proof of Corollary \ref{cor:bps}.

\subsection*{Acknowledgements} The author is indebted with Ursula Hamenst\"{a}dt for her support in the last months. I would like to thank Stefano Francaviglia who encouraged me to submit this notes. I am finally grateful to Lorenzo Ruffoni for his careful reading and helpful comments which led to significant improvements of the exposition.
The first draft of the present note has been written mostly during the conference "Spherical surfaces and related topics" at Cortona, Italy and completed during the $3-$days school "Geometry and dynamics of moduli spaces" at AMS University of Bologna the following week. The author is grateful to the organisers of both conferences.

\section{Translation surfaces with poles}\label{sec:tswp}

\noindent We shall begin by introducing translation surfaces with poles, their periods characters and their groups of translations.

\subsection{Translation surfaces} We have already alluded in the introduction above that translation surfaces may be seen as special examples of branched projective structures. We begin by providing an independent definition of translation surfaces and we shall subsequently explain in subsection \S\ref{ssec:tswpasbps} how these structures actually appear as members of that larger family of structures. 

\smallskip

\noindent Let $X\in\mathcal{M}_g$ be a compact Riemann surface. Let $\Omega(X)$ be the complex vector space of abelian differentials on $X$. In the sequel it will be convenient for us to adopt the following terminology which is completely borrowed from \cite{SG}.

\begin{defn}[Kinds of differentials]\label{def:kindofdiff} An abelian differential $\omega$ on a compact Riemann surface $X$ is said to be of the \textit{first kind} if it is a holomorphic differential. A differential $\omega$ is said to be of the \textit{second kind} if it is a meromorphic differential on $X$ with no more than a finite set of singular points that are poles each of which with zero residue. Finally, a differential $\omega$ is said to be of the \textit{third kind} if it is an arbitrarily meromorphic differential on $X$ with no more than a finite set of singular points (there are no longer restrictions on the residues).
\end{defn}

\begin{rmk}
It is worth noticing that all abelian differential of the first kind are also differentials of the second kind and, in turns, these latter are also differentials of the third kind. This fact can be easily seen by writing on a neighborhood of each point the Laurent series with respect to an appropriate coordinate.
\end{rmk}

\noindent Based on Definition \ref{def:kindofdiff} above, the complex-analytic definition of translation surface is as follow.

\begin{defn}[Translation surface]\label{tswp}
Let $X$ be a compact Riemann surface in $\mathcal{M}_g$. A \emph{translation structure} is the datum of an abelian differential $\omega\in\Omega(X)$. We shall define \textit{translation surface} a pair $(X,\omega)$ where $\omega\in\Omega(X)$. A translation surface is said to be \textit{with poles} if $\omega$ is a differential of second or third kind.
\end{defn}

\noindent An abelian differential $\omega$ has a \textit{zero} of order $m$ at a point $p$ if in a local coordinate neighborhood $(U,z)$ of this point $\omega$ has the form $\omega=f(z)\,dz$, where $f$ is a holomorphic function such that $f(z)=z^m\,g(z)$ with holomorphic such that $g(p)\neq0$. Similarly, an abelian differential of the second or third kind has a \textit{pole} of order $h$ at $p$ if in a coordinate neighborhood of $p$ the differential has the form
\begin{equation}
\omega=\left(\frac{a_{-h}}{z^{h}}+\cdots+\frac{a_{-1}}{z}+g(z)\right)dz,
\end{equation} where $g(z)$ is a holomorphic function on $U$ and $a_{-1}=0$ if $\omega$ is of the second kind. A point which is neither a zero nor a pole is called \textit{regular}.

\begin{rmk}[Gauss-Bonnet condition]\label{rmk:gbcond} On a Riemann surface $X\in\mathcal{M}_g$, let $\omega$ an abelian differential of any kind with $k$ zeros of orders $m_1,\dots,m_k$ and $n$ poles of orders $h_1,\dots,h_n$. The Gauss-Bonnet condition relates the orders of zeros and poles as follow
\begin{equation}
    \sum_{i=1}^k m_i \,-\,\sum_{j=1}^n h_j\,=\,2g-2.
\end{equation}
\end{rmk}


\subsection{Groups of translations} We now focus on mappings between translation surfaces. Let $\textnormal{Diff}^+(S_g)$ be the group of orientation preserving diffeomorphisms of $S_g$. Let $X\in\mathcal{M}_g$ be a compact Riemann surface and let $\omega_1,\,\omega_2\in\Omega(X)$ be two abelian differentials. A diffeomorphism $f:(X,\,\omega_1)\longrightarrow (X,\,\omega_2)$ is said to be a \textit{translation} if $f^*\omega_2=\omega_1$. In particular, $f$ provides an isometry between $(X,\omega_1)$ and $(X,\omega_2)$; in fact a translation in local charts. In this case, notice that the differentials $\omega_1,\,\omega_2$ are necessarily of the same kind.


\begin{defn}
Let $X\in\mathcal{M}_g$ be a compact Riemann surface and let $\omega\in\Omega(X)$ be an abelian differential. The \textit{group of translations} of $(X,\omega)$ is defined as
\begin{equation}\label{eq:transgroup}
    \aut(X,\omega)=\left\{\,f\in \textnormal{Diff}^+(S_g)\,\,\big|\,\, f^*\omega=\omega \,\right\}.
\end{equation}
\end{defn}

\smallskip

\noindent Let $\aut(X)$ denote the group of conformal automorphisms of $X$. The following Lemma is immediate from the definition.

\begin{lem}
Let $(X,\omega)$ be a translation surface possibly with poles. Then $\aut(X,\omega)<\aut(X)$.
\end{lem}

\noindent In what follows, we shall need the following 

\begin{lem}\label{lem:goodquot}
Let $(X,\omega)$ be a translation surface, possibly with poles, and let $G$ its group of translations. Then $(X,\omega)/G$ is a translation surface, possibly with poles.
\end{lem}

\begin{proof}
We begin by observing that, since $G=\aut(X,\omega)\le \aut(X)$, then $Y=X/G$ is a Riemann surface. Let $\pi:(X,\omega)\longrightarrow (X,\omega)/G$ be the regular covering, possibly branched, of degree $d=|G|$ arising from the natural action of $G$ on $(X,\omega)$. As $\omega$ is invariant under the action of $G$, at any point with trivial stabiliser it descends to the quotient $(X,\omega)/G$. Let $q_1,\dots, q_n\in X/G$ be the branched values of $\pi$. Since the covering $\pi$ is regular, all the preimages of any $q_i$ have the same local degree $d_i\ge2$. Let $s_i$ be the cardinality of $\pi^{-1}(q_i)$; notice that $d=s_i\,d_i$ for all $i=1,\dots,n$. For any $i$, let $\{p_1,\dots,p_{i\,s_i}\}=\pi^{-1}(q_i)$. For any $j=1,\dots,s_i$, the points $p_{ij}$ all have non-trivial stabilizer $G_{ij}<G$, where $|G_{ij}|=d_i$. It can be showed that a point with non-trivial stabilizer is always a singularity for $\omega$; it can be either a zero or a pole, of order $m_{ij}$ which is a multiple of $d_i$. Therefore, $\omega$ descends to a differential, say $\xi$, on $Y$ with a singularity of order $m_{ij}/d_i\in\Z$ at $q_i$. In particular, $(X,\omega)/G=(Y,\xi)$ is a translation surface with poles and $\omega=\pi^*\xi$.
\end{proof}

\begin{rmk}
Notice that any regular point has trivial stabilizer or, equivalently, any point with non-trivial stabiliser is a singularity for the differential.
\end{rmk}

\smallskip

\subsection{Period characters} Let $X\in\mathcal{M}_g$ be a compact Riemann surface and let $\omega\in\Omega(X)$ be an abelian differential. Let $\Sigma_\omega=\{\,\textnormal{poles of}\,\omega\,\}$ and let $n=|\Sigma_\omega|$ be its cardinality. According to Definition \ref{def:kindofdiff}, we can notice that $n$ is a finite non-negative integer. A differential $\omega$ restricts to a holomorphic differential on $X\setminus \Sigma_\omega$ and it always determines a representation $\chi:\shomolzn\longrightarrow\C$ defined by integration along closed loops on $S_{g,n}$ as in \eqref{eq:per}. The representation $\chi$ is called the \textit{period character} of $\omega$.

\medskip

\noindent A classical Theorem by Haupt, originally stated in \cite{HO} and recently rediscovered by Kapovich in \cite{KM}, provides necessary and sufficient conditions for a representation to appear as the period character for an abelian differentials of the first kind, \textit{i.e.} holomorphic. 

\begin{hthm}
A representation $\chi:\shomolz\longrightarrow\C$ appears as the period character of some abelian differential of the first kind if and only if the following conditions holds:
\begin{itemize}
    \item[1.] $\textnormal{vol}(\chi)=\displaystyle\sum_{i=1}^g\,\Im\left(\,\overline{\chi(\alpha_i)}\,\chi(\beta_i)\,\right)>0$, where $\{\alpha_1,\beta_1,\dots,\alpha_g,\beta_g\}$ is a symplectic basis of $\,\shomolz$;\\
    \item[2.] if $\,\textnormal{Im}(\,\chi\,)=\Lambda$ is a lattice in $\C$, then $\textnormal{vol}(\chi) \ge 2\cdot\textnormal{vol}\big(\C/\Lambda\big)$.
\end{itemize}
\end{hthm}

\smallskip

\noindent The quantity $\textnormal{vol}(\chi)$ is called \textit{volume} of $\chi$ and it does not depend on the choice of a symplectic basis. This terminology is motivated by the fact that it coincides with the area of the singular Euclidean metric on $S_g$ determined by $\omega$. Haupt's Theorem is specific to the holomorphic case. In fact, there are no obstructions on realising a representation as the period character of some differential of second or third type as recently proved in \cite[Theorem A]{CFG}. More precisely

\begin{thmnn}[\cite{CFG}] 
Every representation $\chi:\shomolzn\longrightarrow\C$ appears as the period character of some abelian differential of the second or third kind  where all the poles are at the punctures.
\end{thmnn}

\smallskip

\noindent The following straightforward observations will be a key fact in the sequel, see Section \S\ref{sec:agtswp}.

\begin{lem}\label{lem:factor}
Let $X\in\mathcal{M}_g$ be a compact Riemann surface and let $\omega\in\Omega(X)$ be an abelian differential of any kind. Let $G=\aut(X,\omega)$ and let $\pi:(X,\omega)\longrightarrow (X,\omega)/G$ be the covering projection. Then the period character $\chi$ of $\omega$ factors through $\pi_*$ in homology.
\end{lem}

\begin{proof}
Let $(X,\omega)$ be a translation surface, possibly with poles, and let $G$ be its group of translations. The group $G$ acts on $(X,\omega)$ by translations and, according to Lemma \ref{lem:goodquot} above, the quotient space $(Y,\xi)$ is a translation surface such that $\omega=\pi^*\xi$. Let $\textnormal{Sing}(\omega)\subset X$ be the set of singularities of $\omega$, that is zeros and poles.
Notice that $\textnormal{Sing}(\omega)$ cannot be empty if $g\ge2$.
Let $n=|\,\textnormal{Sing}(\omega)\,|$.
The covering projection $\pi$ restricts to a regular covering map, say $\pi^{\text{red}}$, defined as
\begin{equation}
    S_{g,\,n}\cong X\setminus\textnormal{Sing}(\omega) \longrightarrow Y\setminus \pi\left( \textnormal{Sing}(\omega)\right)\cong S_{h,\,m}. 
\end{equation}
This covering map induces a mapping  $\pi^{\text{red}}_*$ in homology, that is 
\begin{equation}
    \pi^{\text{red}}_*:\textnormal{H}_1(S_{g,\,n},\,\Z)\longrightarrow \textnormal{H}_1(S_{h,\,m},\,\Z).
\end{equation}
Since $\omega=\pi^*\xi$, then
\begin{equation}
    \int_\gamma \pi^*\xi\, =\,\int_{\pi^{\text{red}}_*(\gamma)} \xi
\end{equation}
for any $\gamma\in\textnormal{H}_1(S_{g,\,n},\,\Z)$ and $\pi^{\text{red}}_*(\gamma)\in\textnormal{H}_1(S_{h,\,m},\,\Z)$ and the desired conclusion follows.
\end{proof}

\section{Automorphism groups of translation surfaces}\label{sec:agtswp} 

\noindent This section is devoted to prove Theorem \ref{thm:main}, its Corollary \ref{cor:main} and Propositions \ref{prop:autchar} and \ref{prop:charlarge}. We begin with the proof of these propositions.

\subsection{Large groups of translations}\label{ssec:large}

In this section we investigate the relationship between automorphisms and period characters. Let $(X,\omega)$ be a translation surface possibly with poles. 

\begin{defn}[Large group of translations]\label{def:lgt}
A translation surface $(X,\omega)$ has a \textit{large group of translations} if $(X,\omega)/\aut(X,\omega)$ is a sphere equipped with a differential $\xi\in\Omega(\cp)$. In this case, we shall also say that $(X,\omega)$ is a \textit{large translation surface}.
\end{defn}

\noindent It is worth mentioning that there are no non-trivial holomorphic differential on $\cp$, see Remark \ref{rmk:gbcond}. Therefore, it directly follows from Definition \ref{def:lgt} and Lemma \ref{lem:goodquot} that a pair $(X,\omega)$ has a large group of translations if and only if $\omega$ is a meromorphic differential.

\begin{proof}[Proof of Proposition \ref{prop:autchar}]

\noindent Let $(X,\omega)$ be a translation surface of genus $g_X\ge2$; where $\omega$ is a differential of any kind. Since its group of translations $\aut(X,\omega)$ is not large, the quotient space $(Y,\xi)$ is a translation surface (possibly with poles) of positive genus; $g_Y\ge1$. The desired result follows as a straightforward application of Riemann-Hurwitz formula. This is the same computation developed by Schlage-Puchta and Weitze-Schmith\"{u}sen in \cite{SPWS} and we include the details for the reader's convenience. 

\smallskip

\noindent Let $\pi:(X,\omega)\to (Y,\xi)$ be the covering projection of degree $\deg(\pi)=|\aut(X,\omega)\,|$. Let $q_1,\dots,q_k\in (Y,\xi)$ be the branched values of $\pi$. Recall that $\pi$ is a regular covering and hence for any $q_i$ all the preimages have the same local degree $d_i$. Finally, let $s_i$ be the cardinality of $\pi^{-1}(q_i)$. Then
\begin{align}
    2g_X-2&=\deg(\pi)\big(2g_Y-2\big)+\sum_{i=1}^k s_i(d_i-1)\\
     &=\deg(\pi)(2g_Y-2)+k\deg(\pi)-\sum_{i=1}^k s_i\\
      &\ge \deg(\pi)\left(2g_Y-2+\frac{k}{2}\right).
\end{align}
In order to maximise $\deg(\pi)$ we need to minimise $\left(2g_Y-2+\frac{k}{2}\right)$. Observe that since $\aut(X,\omega)$ is not large, then $g_Y\ge1$ and $2g_Y-2+\frac{k}{2}\ge0$. The equality holds for $g_Y=1$ and $k=0$ but these conditions would imply $g_X=1$, which is a contradiction because we are assuming $g_X\ge2$. Therefore $\left(2g_Y-2+\frac{k}{2}\right)\ge\frac12$ and hence $|\aut(X,\omega)\,|=\deg(\pi)\le 4g_X-4$ as desired.
\end{proof}

\noindent Notice that $g_X\ge2$ is a necessary condition because there are genuine coverings of a torus to itself in any degree and hence there is no longer an upper bound on the cardinality of $\aut(X,\omega)$. This also serves to show that the largeness property of a group is not intended in the sense of its cardinality but on how big is the Euler characteristic of the quotient space. A group of translations is \textit{large} if the Euler characteristic of the quotient space is as large as possible, namely $2$. 



\subsection{Periods character of large translation surfaces}
We aim to prove Proposition \ref{prop:charlarge}; that is we aim to provide a characterisation of whether a representation $\chi:\shomolzn\longrightarrow \C$ appears as the period character of some large translation surface. According to Proposition \ref{prop:autchar}, this is never the case if $n=0$, that is $\chi$ is determined by a holomorphic differential on a Riemann surface $X\in\mathcal{M}_g$. Let us premise the following lemma whose proof is a routine exercise.

\begin{lem}\label{lem:algprem}
Let $f:G\longrightarrow \Gamma$ be an epimorphism of groups. Let $H\lhd\Gamma$ be a subgroup of finite index. Then $f^{-1}(H)<G$ is normal of finite index. Moreover, $[f^{-1}(H):G]=[H:\Gamma]$.
\end{lem}

\begin{proof}[Proof of Proposition \ref{prop:charlarge}] Let $(X,\omega)$ be a large translation surface and let  $\chi_\omega:\shomolzn\longrightarrow \C$ be its period character. By Definition \ref{def:lgt}, the space $(X,\omega)/\aut(X,\omega)$ is a sphere equipped with an abelian differential $\xi$, necessarily meromorphic. Let $\pi:(X,\omega)\longrightarrow (\cp,\,\xi)$ denote the covering projection. 
Lemma \ref{lem:goodquot} implies $\omega=\pi^*\xi$ and by Lemma \ref{lem:factor} the period character $\chi_\omega$ factors through $\pi_*$ in homology, that is $\chi_\omega=\chi_\xi\circ \pi_*$. Therefore $\textnormal{Im}(\,\chi_\omega\,)\lhd\textnormal{Im}(\,\chi_\xi\,)\lhd\C$ and the first conclusion follows.
\smallskip

\noindent Let us prove the second assertion. Let $\xi\in\Omega(\cp)$ be a meromorphic differential and let $\chi_\xi$ its period character. Let $\Gamma=\textnormal{Im}(\,\chi_\xi\,)$ and let $H\lhd \Gamma$ be a subgroup of finite index. Notice that, being $H\lhd\Gamma\lhd\C$, both groups are abelian. Let $\Sigma_\xi$ be the set of poles of $\xi$, let $m=|\,\Sigma_\xi\,|$ be its cardinality and define $S_{o,\,m}$ as the punctured sphere $\cp\setminus \Sigma_\xi$. Let $\rho_\xi:\pi_1(S_{0,\,m})\longrightarrow \C$ be the representation obtained by precomposing the period character $\chi_\xi$ with the epimorphism $\pi_1(S_{o,\,m})\longrightarrow \textnormal{H}_1(S_{o,\,m},\,\Z)$ and define $K=\rho^{-1}_\xi(H)\,\lhd\, \pi_1(S_{o,\,m})$. By classical theory, there is a surface $S_{g,n}$, with $\pi_1(S_{g,\,n})=K$, and a $G$-regular covering $\pi:S_{g,n}\longrightarrow S_{o,m}$, where $G=N/K$ and $N$ is the normaliser of $K$ in $\pi_1(S_{o,\,m})$.
By Lemma \ref{lem:algprem} above, since $H$ is normal of finite index in $\Gamma$, then $K$ is normal of finite index in $\pi_1(S_{o,\,m})$. Therefore, $N=\pi_1(S_{o,\,m})$ and $\pi$ is a regular $G-$covering with $G=\pi_1(S_{o,\,m})/K$. By Lemma \ref{lem:algprem}, the equality $[\,\pi_1(S_{o,\,m})\,:\,K\,]=[\,\Gamma:H\,]$ holds and then $G\cong\Gamma/H$. The mapping $\pi$ extends to a branched $G$-covering $S_g\longrightarrow \mathbb S^2$ and the pull-back of the natural complex structure of $\cp$ determines a compact Riemann surface $X\in\mathcal{M}_g$. Finally, the abelian differential $\xi\in\Omega(\cp)$ pulls-back to an abelian differential $\omega\in\Omega(X)$ on $X$ of the second or third kind and, by construction, $G\cong\aut(X,\omega)$.
\end{proof}

\noindent As alluded in the introduction, as a consequence of Proposition \ref{prop:charlarge}, a representation $\chi:\shomolzn\longrightarrow\C$ can be realised as the period of some large translation surface if it factors as 
\begin{equation}\label{eq:factor}
    \chi:\shomolzn\overset{\pi_*}{\longrightarrow}\textnormal{H}_1(S_{o,\,m},\Z)\longrightarrow \C,
\end{equation} where the former map is induced by a $G-$covering and the latter is the period character of some abelian differential $\xi\in\Omega(\cp)$. For any such a differential $\xi$, it is an easy matter to verify that the image $\Gamma$ of its period character is a group generated by the residues of $\xi$ at the punctures. This condition leads to examples of representations that does not factor as in \eqref{eq:factor}; \textit{e.g.} a representation $\chi$ such that $\chi(\delta)=0$ for any peripheral loop $\delta\in\shomolzn$ cannot factor as in \eqref{eq:factor} unless it is the trivial one. 


\subsection{Realising finite groups as groups of translations} In the present subsection we aim to prove our Theorem \ref{thm:main} and its Corollary \ref{cor:main}.

\smallskip

\noindent As already mentioned in the introduction, for a compact Riemann surface $X\in\mathcal{M}_g$, where $g\ge2$, its group of conformal automorphisms $\aut(X)$ is a finite group of cardinality at most $84(g-1)$. Moreover, any finite group can be realised as the subgroup of some $\aut(X)$, see \cite{Hur}. This latter fact is well-known in literature and we include a proof here for the sake of completeness.

\begin{prop}\label{prop:realfinitegroups}
Any finite group can be realised as the subgroup of $\aut(X)$ for some compact Riemann surface $X$.
\end{prop}

\begin{proof}
Let $G$ be a finite group generated by $h_1,\dots,h_g$, for some $g\ge2$, and let $\Gamma$ be the group defined as
\begin{equation}\label{eq:gamma}
    \Gamma=\left\langle a_1,b_1,\dots,a_g,b_g\,\,\big|\,\,[a_1,\,b_1]\cdots[a_g,\,b_g]=1\right\rangle.
\end{equation}
\noindent Let $\phi:\Gamma\longrightarrow G$ be the homomorphism defined as $\phi(a_i)=h_i,\,\phi(b_i)=h_i^{-1}$. If $K$ denotes the kernel of $\phi$ then $G\cong \Gamma/K$. By design, $\Gamma$ is the fundamental group of a compact Riemann surface $Y$ of genus $g$. By uniformization theorem, $Y$ is biholomorphic to the upper-half plane $\mathbb H$ by the action of a discrete group of M\"obius transformations isomorphic to $\Gamma$. If we restrict the action to $K \unlhd F$ we obtain a Riemann surface, say $X=\mathbb H/K$. Notice that $X$ is also compact because it is a finite covering of $Y=\mathbb H/\Gamma$. In fact, the index $[K:\Gamma]=|\,G\,|$ is finite. Recall that every automorphism of $X$ is induced by a biholomorphism of $\mathbb H$ normalising $K$. Since $K$ acts trivially on $\mathbb H/K$ it directly follows that $\aut(X)\cong N_{\aut(\,\mathbb H\,)}(\,K\,)/K$. This latter contains $\Gamma/K$ as a subgroup, hence $G$ is isomorphic to a subgroup of $\aut(X)$ as desired.
\end{proof}

\begin{rmk}\label{rmk:thmfirst}
Inspired from the argument above it would be already possible to show the first part part of Theorem \ref{thm:main}. In fact, let $G$ be a finite group generated by $h_1,\dots,h_g$, for some $g\ge2$. Let $\Gamma$ as in \eqref{eq:gamma}, let $\phi:\Gamma\longrightarrow G$ be the homomorphism defined by $\phi(a_i)=h_i,\,\phi(b_i)=h_i^{-1}$ and let $K$ denote as above the kernel of $\phi$. Since $\Gamma$ is defined as the fundamental group of a surface of genus $g$ and since $K\unlhd\Gamma$, then $K$ is also isomorphic to the fundamental group of a closed surface of genus $h>g$. By construction, the natural injection $K\hookrightarrow \Gamma$ determines a $G$-covering of surfaces $\pi:S_h\longrightarrow S_g$. Let $Y\in\mathcal{M}_g$ be a compact Riemann surface and let $\xi\in\Omega(Y)$ be a meromorphic differential of the second kind (hence $\xi$ is not of the first kind). It is a classical result that any compact Riemann surface admits such differentials; see also \cite{CFG} for direct constructions. Let $X\defeq \pi^*Y\in\mathcal{M}_h$ be the pull-back Riemann surface and let $\omega=\pi^*\xi$. By construction $G$ appears as the group of translations of $(X,\omega)$. It is worth mentioning that $G$ does not necessarily coincides with $\aut(X)$ which is generally bigger than $\aut(X,\omega)$. In fact, according to Proposition \ref{prop:realfinitegroups}, the equality holds if and only if $\Gamma=N_{\aut(\,\mathbb H\,)}(\,K\,)$. We shall provide below a different way for realising $G$ in such a way that $G\cong\aut(X)\cong\aut(X,\omega)$ for some pair $(X,\omega)$. 
\end{rmk}

\noindent In his work \cite{GL}, Greenberg extended Hurwitz's result to non-compact Riemann surfaces. More precisely, he proved that \textit{every} finite group $G$ is \textit{isomorphic} to some $\aut(X)$, where $X$ is a, possibly non-compact, Riemann surface of finite type. Recall that for a surface $S$, being of finite type is equivalent to having a finitely generated fundamental group. In \cite[Theorem 6']{GL2}, Greenberg subsequently strengthened his earlier result as follow:

\begin{prop}\label{prop:everyisautx}
Let $Y$ be a compact Riemann surface of genus $g$ and let $G$ be a non-trivial finite group. Then there exists a normal covering $\phi: X\longrightarrow Y$, whose group of covering transformations is isomorphic to $G$, and it is the full group of conformal automorphisms of $X$.
\end{prop}

\noindent The gist of the idea for proving this Proposition is similar to that of Proposition \ref{prop:realfinitegroups}. However, it depends on a delicate construction of maximal Fuchsian groups with a given signature \cite[Theorem 4]{GL2}. Once again, by relying on that Theorem, we include here the proof or Proposition \ref{prop:everyisautx} for the sake of completeness. 

\begin{proof}
Let $Y$ be a compact Riemann of genus $g$. Let $G$ be a non-trivial finite group and choose a set of generators $g_1,\dots,g_k$ such that 
\begin{itemize}
    \item[1.] $g_i$ has order $\nu_i>1$,
    \item[2.] $g_1\cdot\,\cdots\,\cdot g_k=1$,
    \item[3.] $k> 6g-3$ ($k>6$ if $g\le1$).
\end{itemize}
This is always possible by repeating generators more than once, and choosing $g_k$ to satisfy the second condition. According to \cite[Theorem 4]{GL2}, there is a Fuchsian group $\Gamma$ of signature $(g;\,\nu_1,\dots,\nu_k)$ such that $\mathbb H/\Gamma\cong Y$. The group $\Gamma$ is defined as 
\begin{equation}
    \Gamma=\left\langle\, a_1,b_1,\dots,a_g,b_g\,e_1,\dots,\,e_k\,\big|\,\,[a_1,\,b_1]\cdots[a_g,\,b_g]\,e_1\cdots e_k=1,\, e_1^{\nu_1}=1,\dots,e_k^{\nu_k}=1\,\right\rangle.
\end{equation}
\noindent The mapping defined $a_1\longmapsto 1,\,\,b_i\longmapsto 1$ and $e_i\longmapsto x_i$ extends to a epimorphism $\phi:\Gamma\longrightarrow G$. Let $K=\textnormal{ker}(\phi)$. Since the elements $e_j$ and $x_j$ have the same order, the group $K$ is torsion-free and $\Gamma/K\cong G$. Since $[K:\Gamma]=|\,G\,|$ is a finite group, then $X=\mathbb H/K$ is a compact Riemann surface and a finite covering of $Y=\mathbb H/\Gamma$. Let $\aut(X)\cong N_{\aut(\mathbb H)}(\,K\,)/K$ be the group of conformal automorphisms of $X$. Since $K\lhd\Gamma$, then $\aut(X)\cong G$ as desired.
\end{proof}


\noindent Based on this result, we can prove our Theorem \ref{thm:main}; that is we can show that every finite group $G$ appears the group of conformal automorphisms of some Riemann surface $X$ by showing that there always exists a meromorphic (but not holomorphic) differential $\omega$ of the second or third kind such that $\aut(X,\omega)=G$. The first step is to prove the following

\begin{prop}\label{prop:thmatwo}
Let $X$ be a compact Riemann surface of genus $g\ge2$ and let $G$ its group of conformal automorphisms. Then there exists a meromorphic differential $\omega\in\Omega(X)$ with poles of finite orders such that $G=\aut(X,\omega)$.
\end{prop}

\begin{proof}
Let $X$ be a compact Riemann surface, let $G=\aut(X)$ its group of conformal automorphisms and define $Y=X/G$. Let $\pi:X\longrightarrow Y$ denotes as above the covering projection. Let $\xi$ be a meromorphic differential of the second kind on $Y$. Then $\pi^*\xi=\omega$ is an abelian differential on $X$ of the same kind of $\xi$. By construction we obtain $G=\aut(X,\omega)$ as desired.
\end{proof}

\begin{rmk}
If $Y$ has positive genus, one may take $\xi$ as an abelian differential of the first kind (recall that there are no abelian differentials of the first kind in genus zero) and the same argument works. In this case, $\omega=\pi^*\xi$ would be an abelian differential of the first kind on $X$ such that $G=\aut(X,\omega)$. As a straightforward consequence of Riemann-Hurwitz Theorem, if $Y$ has positive genus then $|\,G\,|\le 4g-4$ and hence there is no contradiction with Proposition \ref{prop:autchar} and \cite[Theorem 1.1]{SPWS}.
\end{rmk}

\noindent As a direct consequence of Proposition \ref{prop:everyisautx} and Proposition \ref{prop:thmatwo} we obtain the following

\begin{cor}\label{prop:thmaone}
Every finite group appears as the group of translations of some translation surface with poles.
\end{cor}

\noindent Proposition \ref{prop:thmatwo} and its Corollary \ref{prop:thmaone} imply our main Theorem. Finally, we provide a proof of Corollary \ref{cor:main}. We recall the following

\begin{defn}\label{def:hursfs}
A compact Riemann surface $X$ is called \textit{Hurwitz surface} if its group of conformal automorphisms $\aut(X)$ attains the maximal bound provided by Hurwitz, namely $|\aut(X)\,|=84(g-1)$.
\end{defn}

\begin{rmk}
According to \cite[Chapter \S5]{JW}, the maximal bound of $\aut(X)$ is attained if and only if $X\cong \mathbb H/K$, where $K\lhd\Delta(2,3,7)$ of finite index, where $\Delta(2,3,7)$ denotes the smallest triangle group of hyperbolic type.
\end{rmk} 

\begin{proof}[Proof of Corollary \ref{cor:main}] Let $X$ be a Hurwitz surface. Its group of conformal automorphisms $\aut(X)$ has by definition cardinality $84(g-1)$, where $g\ge2$ is the genus of $X$, and $X/\aut(X)\cong\cp$. Let $\xi$ an abelian differential on $\cp$, necessarily of second or third kind, and let $\omega=\pi^*\xi$. Then $(X,\omega)$ is a translation surface with poles and $|\aut(X,\omega)\,|=84(g-1)$ by construction.
\end{proof}

\noindent We may notice that the period character of $\omega\in\Omega(X)$ is trivial whenever $\xi\in\Omega(\cp)$ is a meromorphic differential of the second kind.


\section{Automorphism groups of branched projective structures}\label{sec:autbps} 

\noindent We finally introduce the notion of branched projective structure on closed surfaces. The purpose of the present subsection is to extend our previous results to this type of structures. In fact, under certain conditions, a translation surface with poles extend to a branched projective structure on the same topological surface. We begin by introducing projective structures on surfaces.

\subsection{Branched projective structures}\label{ssec:bps} Let $S_{g}$ be a surface of genus $g$ and negative Euler characteristic, that is $2-2g<0$. Let $\cp$ be the Riemann sphere and let $\pslc$ be its group of conformal automorphisms acting by M\"{o}bius transformations.
\begin{equation}
    \pslc\,\times\,\cp\longrightarrow \cp, \quad \begin{pmatrix} a & b\\ c & d \end{pmatrix},\, z\longmapsto \frac{az+b}{cz+d}.
\end{equation}

\smallskip

\begin{defn}\label{def:bps}
A \textit{projective structure} $\sigma$ on $S_{g}$ is the datum of a maximal atlas of local $\cp-$charts of the form $z\longmapsto z^{k+1}$, for some $k\ge0$, and transition functions given by M\"{o}bius transformation on their overlapping. The structure $\sigma$ is called \textit{unbranched} if $k=0$ for every local chart otherwise it is called \textit{branched}. A point $p\in S_{g}$ will be called a \textit{branch point} of order $k\in\Z^+$ if any local chart at $p$ is a branched cover of degree $k+1$, that is if it looks like of the form $z\longmapsto z^{k+1}$. Notice that the order of any point does not depend on choice of the local chart and hence $k$ is always well-defined.
\end{defn}


\smallskip 

\noindent A branched projective structure is always specified by a \textit{developing map} $\textnormal{dev}:\widetilde{S}_{g}\longrightarrow \cp$, obtained by analytic continuation of local charts, and by a \textit{holonomy representation} $\rho:\pi_1(S_{g})\longrightarrow \pslc$ such that
\begin{equation}
    \textnormal{dev}\left(\gamma\cdot \widetilde{p}\right)=\rho(\gamma)\cdot \textnormal{dev}\left(\widetilde{p}\right);
\end{equation}
\noindent where $\gamma\in\pi_1(S_{g})$. We refer to \cite{DD} for a nice survey about unbranched projective structures; for more about the geometry of the spaces of these structures see \cite{Far}. Notice that for branched projective structures there is a well-defined notion of angle induced by the underlying conformal structure. The total angle around a regular point is $2\pi$ whereas for a branch point, say $p$, the total angle around it has magnitude $2k\,\pi$,  where $k\ge2$ is the degree of any chart at $p$. Finally, notice that Definition \ref{def:bps} naturally extends to branched projective structure on punctured surfaces $S_{g,n}$.

\begin{ex}\label{ex:ups}
For an open set $\Omega\subset \cp$ preserved by a  $\Gamma<\pslc$ acting freely and properly discontinuously, the quotient space $\Omega/\Gamma$ has a natural unbranched projective structure in which the charts are local inverses of the covering $\Omega\longrightarrow \Omega/\Gamma$. By the classical uniformization theory, any Riemann surface $X$ is of the form $\Omega/\Gamma$ where $\Omega$ is an open subset of $\cp$ and $\Gamma$ is a discrete sub-group of $\pslc$ acting freely and properly discontinuously on $\Omega$. This endows $X$ with a natural unbranched projective structure, here denoted with $\sigma_{X}$, coming from the identification $X\cong\Omega/\Gamma$. For surfaces of genus at least two; in the case $X\cong\mathbb H/\Gamma$ where $\mathbb H$ is the upper half-plane and $\Gamma<\pslr<\pslc$ is cocompact and torsion-free, we call the projective structure such defined as \textit{Fuchsian uniformization}. 
\end{ex}

\begin{ex} Let $X\in\mathcal{M}_g$ be a compact Riemann surface. Given $p_1,\dots,p_n\in X$ and $k_1,\dots,k_n\in\Z^+$, in \cite{TM}, Troyanov showed that there exists on $S_g\setminus\{p_1,\dots,p_n\}$ a unique conformal flat Riemannian metric with constant curvature equal to $0$ and branched points of angle $2\pi(k_i+1)$ at $p_i$ if 
\begin{equation}\label{eq:troy}
   2-2g+\sum_{i=1}^n k_i=0.
\end{equation}
Any structure obtained in this way provide an example of branched projective structures. As we shall see below, translation surfaces also belong to this class of examples.
\end{ex}

\smallskip

\noindent A branched projective structure $\sigma$ on a surface $S_g$ restricts to an unbranched structure, say $\sigma^*$, on a punctured surface $S_{g,n}$. In a nutshell, we consider a branched projective structure $\sigma$ on $S_g$ and we eventually obtain $\sigma^*$ by removing all the branch points. Since transition functions are holomorphic mappings, any branched projective structure determines an underlying complex structure $X^*$ on $S_{g,n}$ obtained by extending the maximal atlas for $\sigma^*$ to a maximal atlas of complex charts. In turns, $X^*$ extends to a compact Riemann surface $X$. Therefore a projective structure always determines a compact Riemann surface.

\smallskip

\noindent Conversely, a projective structure over a complex structure $X^*$ on $S_{g,n}$ can be seen as a choice of a sub-atlas of unbranched charts, say $\sigma^*$, maximal in the sense of Definition \ref{def:bps}. In turns, this latter extends to a, possibly branched, projective structure $\sigma$ on $S_g$ if the holonomy of each puncture is trivial. A maximal atlas of complex charts generally encapsulates several, in fact uncountably many, choices of complex projective atlases. In other words, there are infinitely many branched projective structures with the same underlying Riemann surface $X$. 

\smallskip

\subsection{Translation surfaces as branched projective structures}\label{ssec:tswpasbps} \noindent In section \S\ref{sec:tswp} we have introduced translation surfaces as complex-analytic objects; however they can be defined in a more geometric-flavor language as follows. The interests for translation surfaces is due to the interesting geometric and dynamical properties they exhibit. In the present note we shall limit ourselves to the necessary notions and we refer to the nice surveys \cite{WA} and \cite{Zo} for more details and recent advances.

\smallskip

\noindent Let $X\in\mathcal{M}_g$ be a compact Riemann surface and let $\omega\in\Omega(X)$. Let us denote by $\textnormal{Sing}(\omega)$ the set of singularities of $\omega$, \textit{i.e.} the set of zeros and poles (if any). Any abelian differential $\omega$ determines an Euclidean metric on $X\setminus \textnormal{Sing}(\omega)$. More precisely, in a neighborhood of a point $p\in X\setminus \textnormal{Sing}(\omega)$ we can define a local coordinate as
\begin{equation}
    z(q)=\int_p^{\,q} \omega 
\end{equation}  
in which $\omega=dz$ and the coordinates of two overlapping neighborhoods differ by a translation $z\mapsto z+c$ for some $c\in\mathbb C$, \textit{i.e.} a translation. This Euclidean metric naturally extends to a singular Euclidean metric over the set of zeros of $\omega$. In fact, if $p\in X$ is a zero for $\omega$ of order $k\ge1$ and $U$ is any open neighborhood of $p$, then there exists a local $z$ such that $\omega=z^k\,dz$. The point $p$ is a \textit{branch point} for $(X,\omega)$ because any local chart around it is locally a simple branched $k+1$ covering over $\C$. Let $\Sigma_\omega\subset X$ denotes the set of poles of $\omega\in\Omega(X)$. We have showed that

\begin{lem}\label{lem:equiv}
A translation surface $(X,\omega)$ always determines on $X^*=X\setminus \Sigma_\omega$ a maximal atlas of local complex-valued charts of the form $z\longmapsto z^{k+1}$, for some $k\ge0$, and transition functions given by translations on their overlappings.
\end{lem}

\begin{rmk}
It is worth mentioning that the converse is also true. A maximal atlas of $\C$-valued function charts that differ by translations on their overlapping always defines a pair $(X,\omega)$ on a surface $S_g$. In fact, since the change of coordinates are holomorphic maps there is a well-defined underlying compact Riemann surface $X$. The abelian differential $\omega$ is then obtained by extending the pull-backs of $dz$ via the local charts. In what follows we shall not need to rely on this implication.
\end{rmk} 

\noindent Analytic continuation of any local chart yields a developing map $\textnormal{dev}:\widetilde{S}_g\longrightarrow \C$ equivariant with respect to a holonomy representation $\rho:\pi_1(S_g)\longrightarrow \C$; where the group of translations of $\C$ identifies with $\C$ itself. Since $\C$ is an abelian group, it is not hard to verify that the representation $\rho$ boils down to a representation $\chi:\shomolz\longrightarrow \C$ which coincides with the period character of $\omega$ as defined in \eqref{eq:per}.

\medskip

\noindent As defined, any translation surface, say $(X,\omega)$ corresponding to an abelian differential of the first kind provides an example of branched projective structure $\sigma$ on $S_g$ where the branch points correspond to the zeros of $\omega$. The Gauss-Bonnet condition, seen in Remark \ref{rmk:gbcond}, forces a translation surface to have finitely many branch points each of finite order. Moreover, their orders sum up to $2g-2$. 

\begin{rmk}
The Gauss-Bonnet condition however does not longer hold for branched projective structures. In fact, it is possible to apply some surgeries to increase the number as well as the orders of the branch points without changing the topology of the surface, see \cite[Bubblings]{GKM}. In particular, this surgery applied to a translation surface turns it into a projective structure which is no longer a translation surface, see \cite[Section \S12]{GKM} for more details.
\end{rmk}

\smallskip

\noindent On the other hand, translation surfaces corresponding to abelian differentials of the second or third kind provide examples of branched projective structures on a punctured surface $S_{g,n}$, where $n\ge1$ is the number of poles. The key observation here is that any translation surface with poles determined by a meromorphic differential of the second kind on a compact Riemann surface $X$ always extends to a branched projective structure on $S_g$ which is no longer a translation surface. In fact, let $\omega$ be an abelian differential of the second kind on $X$ and consider a pole $p\in X$ of order $h+2\ge2$. Let $U$ be an open neighborhood of $p$ and choose an appropriate local coordinate $z$ such that
\begin{equation}
    \omega=\frac{dz}{z^{h+2}}\quad \text{ on } U
\end{equation}
\noindent around $p$. Notice that $U$ is biholomorphic to the punctured disk $\mathbb D^*$. By applying the change of coordinate $\zeta=z^{-1}$, the differential $\omega$ has now a zero of order $h$ and a local chart $\varphi:(U,\,\zeta) \longrightarrow \mathbb C$ is $h+1-$fold covering over the puncture disk. Equivalently, the coordinate neighborhood $(U,\,\zeta)$ is biholomorphic to a neighborhood of the vertex of the Euclidean cone of angle $2(h+1)\pi$ to which the conical singularity has been removed. As a direct application of Riemann extension theorem, the local chart $\varphi$ extends over the point $p$ with $\varphi(p)=0\in \mathbb D$. Finally, by changing a change of coordinates again, that is $z=\zeta^{-1}$, we obtain a chart around $p$ modelled at $\infty\in\cp$.

\begin{rmk}
We cannot deduce the geometry around the pole from this model because the change of coordinate $\zeta=\frac1z$ is not a translation and so the geometry has been altered. However it gives a glimpse of what the geometry should be. The mapping $z\mapsto\frac1z$ is an inversion and hence the geometry around a pole is that of the exterior of a compact neighborhood of $0$ in the Euclidean cone $(\C,\, |\,z\,|^{2h}|\,dz\,|^2)$ of angle $2(h+1)\pi$.
\end{rmk}

\smallskip

\noindent As a consequence, for a translation surface, any maximal atlas of  complex-valued charts, whose existence is ensured by Lemma \ref{lem:equiv}, can be enlarged to a maximal atlas of $\cp-$valued charts by adding charts over the poles of $\omega$. The resulting structure is no longer a translation surface because there are charts taking the value $\infty\in\cp$ and hence the associated developing map is now a holomorphic function $\textnormal{dev}:\widetilde{S}_g\longrightarrow \cp$ (and not $\C$). We just showed the following

\begin{lem}\label{lem:exte}
Let $X\in\mathcal{M}_g$ be a compact Riemann surface and $\omega\in\Omega(X)$ be an abelian differential of the second kind. Then $(X,\omega)$ extends to a branched projective structure $\sigma$ on $S_g$.
\end{lem}

\noindent We can now move to show Corollary \ref{cor:bps} which turns out a direct consequence of this latter Lemma along with the results of Section \S\ref{sec:agtswp}.

\smallskip

\subsection{Realising groups as groups of projective automorphisms} We are going to conclude by providing a quick proof of Corollary \ref{cor:bps} based on our previous discussion above. We now introduce the following terminology as done in Section \S\ref{sec:tswp} and \cite[Section \S2]{FR}. Let $\sigma_1,\,\sigma_2$ be branched projective structures on $S_g$ and let $f:\sigma_1\longrightarrow\sigma_2$ be a diffeomorphism. We say that $f$ is \textit{projective} if its restrictions to local projective charts are given by elements in $\pslc$. We say that $\sigma_1$ and $\sigma_2$ are isomorphic if there exists a projective diffeomorphism between them.


\begin{defn}
Let $\sigma$ be a branched projective structure on $S_g$. We define the group of \textit{projective automorphisms of} $\sigma$ to be 
\begin{equation}\label{eq:transgroup2}
    \aut(X, \sigma)=\left\{\,f\in \textnormal{Diff}^+(S_g)\,\,\big|\,\, f\,\text{ is projective for }\sigma\,\right\},
\end{equation} where $X$ denotes the underlying Riemann surface.
\end{defn}

\noindent Notice that $\aut(X,\sigma)$ is also a subgroup of $\aut(X)$ and hence the former is also subject to the Hurwitz bound $84(g-1)$. Our aim here is to show that this bound is sharp by providing examples of branched projective structures that attain the expected maximal bound. The group $\aut(X,\sigma)$ of projective automorphisms does not need to coincide with the overall group $\aut(X)$ and it is generally harder to determine whether the equality
\begin{equation}\label{eq:rhps}
    \aut(X,\sigma)=\aut(X)
\end{equation}
holds. Whenever the equality \eqref{eq:rhps} is satisfied, we shall say that $(X,\sigma)$ is a \textit{relatively Hurwitz projective structure}, see \cite[Definition 2.14]{FR}. A \textit{Hurwitz projective structure} $(X,\sigma)$ is a projective structure such that $|\aut(X)\,|=|\aut(X,\sigma)\,|=84(g-1)$. For \textit{unbranched} projective structures on surfaces of genus $g\ge2$, in \cite{FR} the authors showed that the Fuchsian uniformization of a compact Riemann surface $X$, see Example \ref{ex:ups}, is always a relatively Hurwitz projective structure and, in particular, it is the only one if and only if $X$ is a Galois-Bely\u{\i} curve, see \cite{JW} for more details about these curves. Nevertheless, for any compact Riemann surface $X$ there exists a branched projective structure $\sigma$ such that the equality \eqref{eq:rhps} holds. In fact, Proposition \ref{prop:thmatwo} and Lemma \ref{lem:exte} imply the analogue of the second statement of Theorem \ref{thm:main} for projective structures.

\begin{prop}
For every compact Riemann surface $X$ there exists a branched projective structure $\sigma$ such that $\aut(X)=\aut(X,\sigma)$.
\end{prop}

\noindent The following is the analogue of the first statement of Theorem \ref{thm:main} for projective structures and it is a straightforward consequence of Corollary \ref{prop:thmaone}. 

\begin{prop}
Every finite group appears as the group of projective automorphisms of some branched projective structure.
\end{prop}

\begin{proof}This is a direct consequence of Corollary \ref{prop:thmaone} and hence we adopt the same notation. For a finite group $G$, there exists a finite degree $G-$covering $\pi:S_h\longrightarrow S_g$. Let $(Y,\xi)$ be a translation surface with poles where $\xi$ is an abelian differential of the second kind, see Definition \ref{def:kindofdiff}. Extend $(Y,\xi)$ to a branched projective structure $(Y,\sigma)$ by "filling" the poles with $\cp-$charts. Then use $\pi$ to pull-back the branched projective structure on $S_h$. Let $(X,\,\pi^*\sigma)$ be the resulting structure, then $G=\aut(X,\,\pi^*\sigma)$  holds by construction.\end{proof}

\noindent We finally prove our Corollary \ref{cor:bps} which we restate here for the reader convenience.

\begin{cor}[\,\ref{cor:bps}\,]
Let $X$ be a Hurwitz surface. Then there exists a branched projective structure $\sigma$ on $X$ such that $|\aut(X,\sigma)\,|=84(g-1)$.
\end{cor}

\begin{proof}[Proof of Corollary \ref{cor:bps}]
Let $X$ be a Hurwitz surface of genus $g\ge2$, see Definition \ref{def:hursfs}, and let $\aut(X)$ its group of conformal automorphisms. Let $\pi:X\longrightarrow X/\aut(X)\cong\cp$ be the covering projection. The natural projective structure on $\cp$ pulls back to a branched projective structure $\sigma$ on $X$. By construction we have $|\aut(X)\,|=|\aut(X,\,\sigma)\,|=84(g-1)$ as desired. Alternatively, let $\xi$ be an abelian differential of the second kind on $\cp$. By Corollary \ref{cor:main}, $(X,\omega)$ is a translation surface with $|\aut(X,\omega)\,|=84(g-1)$. Notice that $\omega\in\Omega(X)$ is a differential of the second kind. Then use Lemma \ref{lem:exte} for extending $(X,\omega)$ to a branched projective structure $(X,\sigma)$. By construction $|\aut(X,\,\sigma)\,|=84(g-1)$.
\end{proof}

\smallskip

\noindent Every Hurwitz projective structure obtained in this way enjoys particular geometric features. In the first place, by construction, they all carry a singular Euclidean metric away from a finite set of isolated points corresponding to those which are modelled at $\infty\in\cp$. Secondly, they all have trivial holonomy. This is in fact a mere consequence of our construction. It would be interesting to determine whether there are branched projective structures with \textit{non-trivial} holonomy. For unbranched structures, we already know from \cite{FR} that the Fuchsian uniformization of a Hurwitz surface is always a Hurwitz projective structure and they are known to have non-trivial discrete holonomy. In addition, we may also wonder whether there are Hurwitz projective structures that carry some global singular Riemannian metric such as hyperbolic or spherical metrics. From \cite{SPWS}, we can deduce no Hurwitz projective structure can arise from a translation surface determined by a holomorphic differential as its group of translation has never more than $4g-4$ translations (in this case there are no points modelled at $\infty\in\cp$).

\medskip


\bibliographystyle{amsalpha}
\bibliography{atswp}

\end{document}